\documentclass[12pt]{article}
\usepackage{amsfonts}
\usepackage{mathrsfs}
\usepackage{color}
\usepackage{amssymb,amsmath,mathrsfs, amsthm}
\usepackage{palatino,cite}
\usepackage{graphicx}
\usepackage{mathtools}
\usepackage{hyperref,cases,anysize,enumerate}

\numberwithin{equation}{section}

\newtheorem{theorem}{Theorem}[section]
\newtheorem{lemma}{Lemma}[section]
\newtheorem{corollary}{Corollary}[section]
\newtheorem{proposition}{Proposition}[section]
\theoremstyle{definition}
\newtheorem{definition}{Definition}[section]
\newtheorem{remark}{Remark}[section]
\newtheorem{example}{Example}[section]

\theoremstyle{remark}

\date{}
\begin{document}

\title{Schwarz lemma from a K\"ahler manifold into a complex Finsler manifold}
\author{Jun Nie (jniemath@126.com)\\
School of Mathematical Sciences, Xiamen
University\\ Xiamen 361005, China\\
Chunping Zhong (zcp@xmu.edu.cn)\\
School of Mathematical Sciences, Xiamen
University\\ Xiamen 361005, China
}

\date{}
\maketitle

\begin{abstract}
 Suppose that $M$ is a K\"ahler manifold with a pole such that its holomorphic sectional curvature is bounded from below by a constant and its radial sectional curvature is also bounded from below. Suppose that $N$ is a strongly pseudoconvex complex Finsler manifold such that its holomorphic sectional curvature is bounded from above by a negative constant. In this paper, we establish a Schwarz lemma for holomorphic mappings $f$ form $M$ into $N$. As applications, we obtain a Liouville type rigidity result for holomorphic mappings $f$ from $M$ into $N$, as well as a rigidity theorem for bimeromorphic mappings from  a compact complex manifold into a compact complex Finsler manifold.
\end{abstract}



\section{Introduction}

In analysis of one complex variable, the classical Schwarz lemma states that a holomorphic mapping from the open unit disk into itself decreases the Poincar\'e metric. It is known that the classical Schwarz lemma plays an important role in proving the Riemannian mapping theorem and in various topics of geometric function theory in complex analysis.
In 1938, Ahlfors \cite{ahlfors} gave a key generalization of the classic Schwarz lemma for holomorphic mappings from the unit disk into a Riemann surface that admits a Hermitian metric with Gauss curvature bounded from above by a negative constant, which open the door of generalizing Schwarz lemma from the viewpoint of differential geometry. In 1957-1958, Look \cite{La,Lb} gave a systematic study of Schwarz lemma and analytic invariants on the classic domains, from the viewpoints of both function theoretic and differential geometric.

The Schwarz lemma has become a powerful tool in geometry and analysis ever since  Yau's seminal paper \cite{yau2} which pushed this classic result in complex analysis to manifolds.
The general theme of the lemma goes something like this: given a holomorphic map $f$ from a complete complex manifold $M$ into a target complex manifold $N$, assume that $M$ has lower curvature bound $K_1$ and $N$ has upper curvature bound by a negative constant $K_2<0$. Then the pull-back via $f$ of the metric of $N$  is dominated by a multiple (which is typically in the form $\frac{K_1}{K_2}$) of the metric of $M$, with the multiple given by the curvature bounds. This type of results immediately imply Liouville type rigidity results when the multiple becomes zero.
 In Yau's original result \cite{yau2}, both $M$ and $N$ are assumed to be K\"ahler manifolds, where $M$ has holomorphic bisectional curvature bounded from below and $N$ has holomorphic sectional curvature bounded from above by a negative constant. Shortly after, Royden \cite{royden} realized that the curvature assumption on the domain could be reduced to holomorphic sectional curvature utilizing the symmetry of curvature tensor of K\"ahler metrics. Since then, various generalizations were made to Hermitian and almost Hermitian cases, we refer to Chen-Cheng-Lu \cite{chen}, Greene-Wu \cite{greene},  Lu-Sun \cite{LS}, Liu \cite{Liu}, Zuo \cite{zuo}, Tosatti \cite{tosatti}, Wu-Yau \cite{wy}, Yang-Zheng \cite{yang}, Ni \cite{Nia,Nib}, and many others.




Finsler geometry is Riemannian geometry without quadratic restrictions. Complex Finsler geometry is Hermitian geometry without Hermitian-quadratic restrictions which contains Hermitian geometry as its special case.
It is known that for any complex manifold, there are natural intrinsic pseudo-metrics, i.e., the Kobayashi and Carath$\acute{\mbox{e}}$odory pseudo-metrics.  In general, however, they are only complex Finsler metrics in nature (see \cite{abate}). 
A natural question in complex Finsler geometry one may ask is whether it is possible to generalize Schwarz lemma to more general differential metric spaces, i.e., to establish  Schwarz lemmas for holomorphic mappings between complex Finsler manifolds. In \cite{shen}, Shen and Shen obtained a Schwarz lemma from a compact complex Finsler manifold with holomorphic sectional curvature bounded from below by a negative constant into another complex Finsler manifold with holomorphic sectional curvature bounded above by a negative constant.  In the case that the source manifold is non-compact, Wan \cite{wan}
obtained a Schwarz lemma from a complete Riemann surface with curvature bounded from below by a constant into a complex Finsler manifold with  holomorphic sectional curvature bounded from above by a negative constant. While the general case when the source manifold is a complete non-compact strongly pseudoconvex complex Finsler manifold is still open.
It seems that the method used in \cite{wan} does not work when the source manifold has complex dimension $\geq 2$.

As a first step towards the above question in this paper we generalize Schwarz lemma to the case when the source manifold $M$ is a complete K\"ahler manifold and the target manifold $N$ is a complex manifold endowed with a strongly pseudoconvex complex Finsler metric $G:T^{1,0}N\rightarrow [0,+\infty)$ in the sense of Abate and Patrizio \cite{abate}. The main results are as follows.

\begin{theorem}\label{mth}
 Suppose that $(M,ds_M^2)$ is a complete K\"ahler manifold with holomorphic sectional curvature bounded from below by a constant $K_1$ and sectional curvature bounded from below, while $(N,H)$ is a strongly pseudoconvex complex Finsler manifold with holomorphic sectional curvature of the Chern-Finsler connection bounded from above by a constant $K_2<0$. Then any holomorphic map $f$ from $M$ into $N$ satisfies
 \begin{equation}
(f^*H)(z;dz) \leq \frac{K_1}{K_2}ds_M^2.
\end{equation}
\end{theorem}

The following example shows that there are lots of strongly pseudoconvex complex Finsler manifolds satisfying the assumption of $(N,H)$ in the above Theorem \ref{mth}. We refer the reader to the proof of Theorem \ref{example} for more details.
\begin{example}Let $N:=\mathbb{B}^n(\ell)=\{\|z\|^2<\ell^2\}$ be an open ball in $\mathbb{C}^n(n\geq 2)$ such that $\ell<\frac{1}{b}$ with $b$ an arbitrary positive constant. Let
$$
H(z;v)=\|v\|^2\exp\Big\{a\|z\|^2+b\frac{|\langle z,v\rangle|^2}{\|v\|^2}\Big\},\quad \forall z\in N,\forall 0\neq v\in T_z^{1,0}N,
$$
where $a$ is an arbitrary positive constant.
Then $H:T^{1,0}N\rightarrow \mathbb{R}^+$ is a non-Hermitian quadratic strongly pseudoconvex complex Finsler metric with holomorphic sectional curvature bounded from above by a negative constant.
\end{example}

As an application, we obtain the following theorem which generalizes a result of Yang and Zheng \cite{yang} from Hermitian to Finsler setting.
\begin{theorem}
Let $M$ be a compact complex manifold of complex dimension $n\geq 2$ which admits a strongly pseudoconvex complex Finsler metric $G$  with negative holomorphic curvature. Let $N$ be a compact complex manifold of the same complex dimension $n$ which admits a holomorphic fibration $f:N \rightarrow Z$, where a generic fiber is a compact K\"ahler manifold with holomorphic sectional curvature bounded form below by a non-negative constant. Then $M$ cannot be bimeromorphic to $N$.
\end{theorem}
\section{Perliminaries}
\setcounter{equation}{0}
\noindent

Let $M$ be a complex manifold of complex dimension $n$. Let $\{z^1,\cdots,z^n\}$ be a set of local complex coordinates, and let $\{\frac{\partial}{\partial z^{\alpha}}\}_{1 \leq \alpha \leq n}$ be the corresponding natural frame of $T^{1,0}M$. So that any non-zero element in $\tilde{M}=T^{1,0}M \setminus \{\text{zero section}\}$ can be written as
$$v=v^{\alpha}\frac{\partial}{\partial z^{\alpha}} \in \tilde{M},$$
where we adopt the summation convention of Einstein. In this way, one gets a local coordinate system on the complex manifold $\tilde{M}$:
$$(z;v)=(z^1,\cdots,z^n;v^1,\cdots,v^n).$$
\begin{definition}(\cite{abate,kobayashi})\label{d-2.1}
A complex Finsler metric $G$ on a complex manifold $M$ is a continuous function $G:T^{1,0}M \rightarrow [0,+\infty)$ satisfying

(i) $G$ is smooth on $\tilde{M}:=T^{1,0}M\setminus\{\mbox{zero section}\}$;

(ii) $G(z;v) \geq 0$ for all $v \in T_z^{1,0}M$ with $z \in M$ and $v \in \pi^{-1}(z)$, and $G(z;v)=0$ if and only if $v=0$;

(iii) $G(z;\zeta v)=|\zeta|^2G(z;v)$ for all $(z;v) \in T^{1,0}M$ and $\zeta \in \mathbb{C}$.
\end{definition}

\begin{definition}[\cite{abate}]
A complex Finsler metric $G$ is called strongly pseudoconvex if the Levi matrix
\begin{equation*}
(G_{\alpha \overline{\beta}})= \Big(\frac{\partial^2 G}{\partial v^{\alpha}\partial \overline{v}^{\beta}}\Big)
\end{equation*}
is positive definte on $\tilde{M}$.
\end{definition}
\begin{remark}[\cite{abate}]
Any $C^\infty$ Hermitian metric on a complex manifold $M$ is naturally a strongly pseudoconvex complex Finsler metric. Conversely,
if a complex Finsler metric $G$ on a complex manifold $M$ is $C^\infty$ over the whole holomorphic tangent bundle $T^{1,0}M$, then it is necessary a $C^\infty$ Hermitian metric. That is, for any $(z;v)\in T^{1,0}M$
$$
G(z;v)=g_{\alpha\overline{\beta}}(z)v^\alpha\overline{v}^\beta
$$
for a $C^\infty$ Hermitian tensor $g_{\alpha\overline{\beta}}$ on $M$.
For this reason, in general the non-trivial (non-Hermitian quadratic) examples of complex Finsler metrics are only required to be smooth over the slit holomorphic tangent bundle $\tilde{M}$.
\end{remark}

\begin{remark}\label{ff}
Taking a vector $v=dz^\alpha\frac{\partial}{\partial z^\alpha}\in T_z^{1,0}M$, and using the $(1,1)$-homogeneity property (iii) of $G$, we have
$$
G(z;dz)=G_{\alpha\overline{\beta}}(z;dz)dz^\alpha\overline{dz^\beta}.
$$
Namely, the first fundamental form $ds_M^2$ of a strongly pseudoconvex complex Finsler metric $G:T^{1,0}M\rightarrow \mathbb{R}$ on a complex manifold $M$ can expressed as
$$ds_M^2=G_{\alpha\overline{\beta}}(z;dz)dz^\alpha\overline{dz^\beta},$$
which in general is not a Hermitian quadratic form of $dz=(dz^1,\cdots,dz^n)$.
\end{remark}

In the following, we use the notions in Abate and Patrizio \cite{abate}. We shall denote by indexes like $\alpha, \overline{\beta}$ and so on the derivatives with respect to the $v$-coordinates; for instance,
\begin{equation*}
G_{\alpha \overline{\beta}}= \frac{\partial^2 G}{\partial v^{\alpha}\partial \overline{v}^{\beta}}.
\end{equation*}
On the other hand, the derivatives with respect to the $z$-coordinates will be denoted by indexes after a semicolon; for instance,
\begin{equation*}
G_{;\mu \overline{\nu}}= \frac{\partial^2 G}{\partial z^{\mu}\partial \overline{z}^{\nu}}\quad \text{or}\quad G_{\alpha;\overline{\nu}}= \frac{\partial^2 G}{\partial \overline{z}^{\nu}\partial v^\alpha}.
\end{equation*}

Using the projective map  $\pi:T^{1,0}M\rightarrow M$, which is a holomorphic mapping, one can define the holomorphic vertical bundle
\begin{equation*}
\mathcal{V}^{1,0}:=\ker{d\pi} \subset T^{1,0}\tilde{M}.
\end{equation*}
It is obvious that $\{\frac{\partial}{\partial v^1},\cdots,\frac{\partial}{\partial v^n}\}$ is a local frame for $\mathcal{V}^{1,0}$.

The complex horizontal bundle of type $(1,0)$ is a complex subbundle $\mathcal{H}^{1,0} \subset T^{1,0}\tilde{M}$ such that
\begin{equation*}
T^{1,0}\tilde{M}=\mathcal{H}^{1,0} \oplus \mathcal{V}^{1,0}.
\end{equation*}
Note that $\{\delta_1,\cdots,\delta_n\}$ is a local frame for $\mathcal{H}^{1,0}$, where
\begin{equation*}
\delta_{\alpha}= \partial_\alpha-\Gamma_{;\alpha}^{\beta}\dot{\partial}_\beta, \quad \Gamma_{;\alpha}^{\beta}:= G^{\beta \overline{\gamma}}G_{\overline{\gamma};\alpha}.
\end{equation*}
Here and in the following, we write $\partial_\alpha:=\frac{\partial}{\partial z^\alpha}$ and $\dot{\partial}_\beta:=\frac{\partial}{\partial v^\beta}$.

For a holomorphic vector bundle whose fiber metric is a Hermitian metric, there is naturally associated a unique complex linear connection (the Chern connection or Hermitian connection) with respect to which the metric tensor is parallel. Since each strongly pseudoconvex complex Finsler metric $G$ on a complex manifold $M$ naturally induces a Hermitian metric on the complex vertical bundle $\mathcal{V}^{1,0}$. It follows that there exists a unique good complex vertical connection $D:\mathcal{X}(\mathcal{V}^{1,0})\rightarrow \mathcal{X}(T_{\mathbb{C}}^\ast\tilde{M}\otimes \mathcal{V}^{1,0})$ compatible with the Hermitian structure in $\mathcal{V}^{1,0}$. This connection is called the Chern-Finsler connecction (see \cite{abate}). The connection $1$-form is
\begin{equation*}
\omega_{\beta}^{\alpha}:=G^{\alpha \overline{\gamma}}\partial G_{\beta \overline{\gamma}}= \Gamma_{\beta;\mu}^{\alpha}dz^{\mu}+\Gamma_{\beta \gamma}^{\alpha} \psi^{\gamma},
\end{equation*}
where
\begin{equation*}
 \Gamma_{\beta;\mu}^{\alpha} =G^{\overline{\tau}\alpha}\delta_{\mu}(G_{\beta \overline{\tau}}),\quad\Gamma_{\beta \gamma}^{\alpha}=G^{\overline{\tau}\alpha}G_{\beta \overline{\tau} \gamma},\quad \psi^{\alpha}= dv^{\alpha}+\Gamma_{;\mu}^{\alpha}dz^{\mu}.
\end{equation*}
The curvature form of the Chern-Finsler connection $D$ associated to $G$ is given by
\begin{equation*}
\Omega^{\alpha}_{\beta}:=R^{\alpha}_{\beta;\mu\overline{\nu}}dz^{\mu}\wedge d\overline{z}^{\nu}+R^{\alpha}_{\beta \delta;\overline{\nu}}\psi^{\delta} \wedge d\overline{z}^{\nu}+R^{\alpha}_{\beta \overline{\gamma};\mu}dz^{\mu}\wedge \overline{\psi^{\gamma}}+
R^{\alpha}_{\beta \delta \overline{\gamma}}\psi^{\delta}\wedge \overline{\psi^{\gamma}},
\end{equation*}
where
\begin{eqnarray*}
R^{\alpha}_{\beta;\mu\overline{\nu}}&=& -\delta_{\overline{\nu}}(\Gamma^{\alpha}_{\beta;\mu})-\Gamma^{\alpha}_{\beta\sigma}\delta_{\overline{\nu}}(\Gamma^{\sigma}_{;\mu}),\\
R^{\alpha}_{\beta\delta;\overline{\nu}} &=& -\delta_{\overline{\nu}}(\Gamma^{\alpha}_{\beta\delta}),\\
R^{\alpha}_{\beta\overline{\gamma};\mu}&=&-\dot{\partial}_{\overline{\gamma}}(\Gamma^{\alpha}_{\beta;\mu})-\Gamma^{\alpha}_{\beta\sigma}\Gamma^{\sigma}_{\overline{\gamma};\mu},\\
R^{\alpha}_{\beta\delta\overline{\gamma}}&=&-\dot{\partial}_{\overline{\gamma}}(\Gamma^{\alpha}_{\beta\delta}).
\end{eqnarray*}

 If $(M,h)$ is a Hermitian manifold. Under a local holomorphic coordinate system $(z^1,\cdots,z^n)$ on $M$, the curvature tensor of the Chern connection of $(M,h)$ has components
\begin{equation*}
R_{\alpha \bar{\beta}\gamma \bar{\delta}}= -\frac{\partial^2 h_{\gamma\bar{\delta}}}{\partial z^{\alpha} \partial {\bar{z}}^{\beta}} +h^{\bar{\lambda}\rho}\frac{\partial h_{\gamma\bar{\lambda}}}{\partial z^{\alpha}}\frac{\partial h_{\rho \bar{\delta}}}{\partial \bar{z}^{\beta}}.
\end{equation*}

For a non-zero tangent vector $v\in T^{1,0}_zM$, one defines the holomorphic sectional curvature $K_h$ of $h$  along $v$ as follows (see \cite{zheng}).
\begin{equation*}
K_h(v)=\frac{2R_{\alpha \bar{\beta}\gamma\bar{\delta}}v^{\alpha}\bar{v}^{\beta}v^{\gamma}\bar{v}^{\delta}}{h^2(v)}.
\end{equation*}
\begin{definition}(\cite{abate})
Let $\mu =h d\zeta \otimes d \bar{\zeta}$ be a Hermitian metric defined in a neighborhood of the origin in $\mathbb{C}$. Then the Gaussian curvature $K(\mu)(0)$ of $\mu$ at the origin is given by
\begin{equation*}
K(\mu)(0) = -\frac{1}{2h(0)}(\Delta \log h)(0),
\end{equation*}
where $\Delta$ denotes the usual Laplacian
\begin{equation*}
\Delta u = 4 \frac{\partial^2 u}{\partial \zeta \partial \bar{\zeta}}.
\end{equation*}
\end{definition}
The following lemma shows that the holomorphic sectional curvature $K_h(v)$ of a Hermitian metric $h$ at a point $z\in M$ in the direction $v$ can be realized by the Gaussian curvature of the induced metric on a $1$-dimensional complex submanifold $S$ passing $z$ and tangent in the direction $v$.
\begin{lemma}(\cite{wuh})\label{l-2.1}
Let $(M,h)$ be a Hermitian manifold, and let $v$ be a unit tangent vector to $M$ at $z$. Then there exists an imbedded $1$-dimensional complex submanifold $S$ of $M$ tangent to $v$ such that the Gaussian curvature of $S$ at $z$
relative to the induced metric equals the holomorphic sectional curvature $K_h$ of $h$ at $z$ in the direction $v$.
\end{lemma}

For a strongly pseudoconvex complex Finsler metric $G$ on $M$,  one can also introduce the notion of holomorphic sectional curvature.
\begin{definition}[\cite{abate}]
Let$(M,G)$ be a strongly pseudoconvex complex Finsler metric on a complex manifold $M$, and take $v \in \tilde{M}$. Then the holomorphic sectional curvature $K_G(v)$ of $G$ along a non zero tangent vector $v$ is given by
\begin{equation*}
K_G(v)=\frac{2}{G(v)^2}\langle\Omega(\chi,\bar{\chi})\chi,\chi\rangle_v.
\end{equation*}
where $\chi=v^\alpha\delta_\alpha$ is the complex radial horizontal vector field and $\Omega$ is the curvature tensor of the Chern-Finsler connection associated to $(M,G)$.
\end{definition}

In complex Finsler geometry, Abate and Patrizio\cite{abate} (see also Wong and Wu\cite{wong}) proved that the holomorphic sectional curvature of $G$ at a point $z \in M$ along a tangent direction $v \in T_z^{1,0}M$ is the maximum of the Gaussian curvatures of the induced Hermitian metrics among all complex curves in $M$ which pass through $z$ and tangent at $z$ in the direction $v$.
\begin{lemma}(\cite{abate,wong})\label{l-2.2}
Let $(M,G)$ be a complex Finsler manifold, $v \in T^{1,0}_zM$ be a nonzero tangent vector tangent at a point $z \in M$. Let $\mathcal{C}$ be the set of complex curves in $M$ passing through $z$ which are tangent to $v$ at $z$. Then the holomorphic sectional curvature $K_G(v)$ of $G$ satisfies the condition
\begin{equation*}
K_G(v)=\max_{S \in \mathcal{C}}K(S)(z),
\end{equation*}
where $K(S)$ is the Gaussian curvature of the complex curve $S$ with the induced metric.
\end{lemma}

\section{Estimation of distance function on Riemannian manifolds}

In this section, we follow the notations in \cite{cheeger}. We introduce the definitions of Hessian and Morse index form in Riemannian geometry. Then we obtain an estimation of the distance function as well as an equality which establishes a relationship between the Hessian of the distance function and the Morse index form on a Riemannian manifold, see Proposition \ref{p-3.1}. Base on this, we obtain an inequality which relates the real Hessian of the distance function and the radial sectional curvature of the Riemannian metric, see Theorem \ref{t-3.2}.

Note that the Hessian of a smooth function $f$ on a Riemannian manifold $(M,g)$ (see \cite{zheng}) is defined by
$$H(f)(X,Y)=X(Yf)-(\nabla_XY)f$$
for any two vector fields $X,Y$ on $M$, where $\nabla$ is the Levi-Civita connection on $(M,g)$.

It is easy to see that $H(f)$ is a symmetric 2-tensor, i.e.,
$$H(f)(Y,X)=H(f)(X,Y)\quad \mbox{and} \quad H(f)(hX,Y)=hH(f)(X,Y)$$
for any smooth function $h$ on $M$.

Now we introduce the second variation of the length integral on a Riemannian manifold, we refer to \cite{cheeger} (see also \cite{zheng}). Let $\gamma:[0,r] \rightarrow M$ be a geodesic, and let $\alpha: Q \rightarrow M$ be a smooth map, where $Q$ is the rectangular solid $[0,r]\times (-\varepsilon,\varepsilon)\times(-\delta,\delta)$ and $\alpha(t,0,0)=\gamma(t)$ for $t\in[0,r]$. This means that $\alpha$ is a 2-parameter variation of the geodesic $\gamma$. Let's take a look at the arc-length function obtained by  successively differentiated with respect to these two parameters. Let $L(s,w)$ be the arc length of the curve $t \rightarrow \alpha(t,s,w)$. Denote $T:=\alpha_\ast(\frac{\partial}{\partial t})$. Then
$$L(s,w)=\int_0^r \|T\|dt,$$
where $\|T\|^2=\langle T |T \rangle$  and $\langle\cdot | \cdot\rangle$ is the Riemannian inner product induced by $g$.

Assume from now on that $\|\dot{\gamma}\|\equiv1$. This means that $\gamma:[0,r] \rightarrow M$ is a normal geodesic. In the following we denote
$S:=\alpha_\ast(\frac{\partial}{\partial s})$ and $W:=\alpha_\ast(\frac{\partial}{\partial w}).$
\begin{theorem}(\cite{cheeger})
Let $(M,g)$ be a Riemannian manifold. Take a normal geodesic $\gamma:[0,r] \rightarrow M$, and let $\alpha:[0,r]\times (-\varepsilon,\varepsilon)\times(-\delta,\delta) \rightarrow M$ be a two-parameter variation of the geodesic $\gamma$. Then
\begin{equation*}\begin{split}
\frac{\partial^2L}{\partial w \partial s}\Big|_{(0,0)}
&=\langle \nabla_WS|T\rangle|_0^r+\int_0^r\{\langle \nabla_TS|\nabla_TW\rangle-\langle R(W,T)T|S \rangle-T\langle S|T\rangle T\langle W|T\rangle\}dt,
\end{split}\end{equation*}
where $R$ is the Riemannian curvature tensor on $(M,g)$.
\end{theorem}
If $\alpha$ is a one-parameter variation of the geodesic $\gamma$, we have the following corollary.
\begin{corollary}(\cite{zheng})\label{c-3.1}
Let $(M,g)$ be a Riemannian manifold. Take a normal geodesic $\gamma:[0,r] \rightarrow M$, and let $\alpha:[0,r]\times (-\varepsilon,\varepsilon) \rightarrow M$ be a one-parameter variation of the geodesic $\gamma$. Then
\begin{equation*}\begin{split}
\frac{d^2L}{d^2 s}\Big|_{s=0}
=\langle \nabla_SS|T\rangle|_0^r+\int_0^r\{\langle \nabla_TS|\nabla_TS\rangle-\langle R(S,T)T|S \rangle-|T\langle S|T\rangle|^2\}dt.
\end{split}\end{equation*}
In particular, if the variation $\alpha$ is fixed, we have
\begin{equation*}\begin{split}
\frac{d^2L}{d^2 s}\Big|_{s=0}
=\int_0^r\{\langle \nabla_TS|\nabla_TS\rangle-\langle R(S,T)T|S \rangle-|T\langle S|T\rangle|^2\}dt.
\end{split}\end{equation*}
\end{corollary}
\begin{remark}
 It follows from Chapter one in \cite{cheeger} that a vector field $S$ along $\gamma$ is a Jacobi field. By Proposition 1.14 in \cite{cheeger},  $\langle S|T\rangle_{\gamma}$ is a constant. Therefore,
\begin{equation*}\begin{split}
\frac{d^2L}{d^2 s}\Big|_{s=0}
=\int_0^r\{\langle \nabla_TS|\nabla_TS\rangle-\langle R(S,T)T|S \rangle\}dt.
\end{split}\end{equation*}
\end{remark}

A vector field $J$ along $\gamma$ is called a Jacobi field if it satisfies the following equation:
$$\nabla_T\nabla_T J- R(T,J)T=0,$$
where $T=\dot{\gamma}$.

The set of all Jacobi fields along $\gamma$ will be denoted by $\mathcal{J}(\gamma)$. A proper Jacobi field is $J \in \mathcal{J}(\gamma)$ such that
$$\langle J|T\rangle \equiv0.$$
Denote by $\mathcal{J}_0(\gamma)$ the set of all proper Jacobi fields along $\gamma$.

Let $\gamma:[0,r] \rightarrow M$ be a normal geodesic in a Riemannian manifold $(M,g)$. We shall denote by $\mathcal{X}[0,r]$ the space of all piecewise smooth vector fields $\xi$ along $\gamma$ such that
$$\langle \xi |T \rangle \equiv 0,$$
where $T=\dot{\gamma}$.
\begin{definition}\label{d-3.1}
The Morse index form $I=I_0^r: \mathcal{X}[0,r]\times \mathcal{X}[0,r] \rightarrow \mathbb{R}$ of the normal geodesic $\gamma:[0,r] \rightarrow M$ is the symmetric bilinear form
$$I(\xi,\eta)=\int_0^r[\langle \nabla_T\xi|\nabla_T\eta\rangle-\langle R(T,\eta)\xi|T\rangle]dt,$$
for all $\xi, \eta \in \mathcal{X}[0,r]$, where $T=\dot{\gamma}$.
\end{definition}

Now we consider the relation between Hessian of the distance function $\rho$ and the second variation formula of the length integral on a Riemannian manifold. Firstly, we introduce the definition of a pole.
\begin{definition}(\cite{greene})\label{d-3.2}
A point $p$ in a Riemannian manifold $(M,g)$ is called a pole if the exponential map $\exp_p:T_pM \rightarrow M$ is a diffeomorphism.
\end{definition}

Given a Riemannian manifold $M$ with a  pole $p$, the radial vector field is the unit vector field $\partial$ defined on $M-\{p\}$, such that for any $x\in M-\{p\}$, $\partial(x)$ is the unit vector tangent to the unique geodesic joining $p$ to $x$ and pointing away from $p$. A plane $\pi$ in $T_xM$ is called a radial plane if $\pi$ contains $\partial(x)$. By the radial sectional curvature of a Riemannian manifold $(M,g)$ we mean the restriction of the sectional curvature function to all the radial planes, we refer to \cite{greene} for more details.
Note that if $M$ possess a pole $p$, then it is complete \cite{greene}. In this case, we denote the distance function from $p$ to $x$ by $\rho(x)$. It is obvious that $\rho(x)$ is smooth on $M\setminus\{p\}$.

The following result (part of which) was actually appeared in Greene and Wu \cite{greene}. We single it out and give a brief proof here since we need it to prove   Theorem \ref{t-3.2}.
\begin{proposition}\label{p-3.1}
Let $(M,g)$ be a Riemannian manifold with a pole $p$. Let $\gamma:[0,r]\rightarrow M$ be a normal geodesic with $\gamma(0)=p$ and $\dot{\gamma}=T$. Then
\begin{equation*}
H(\rho)(X,X)=\int_0^r[\langle \nabla_TJ|\nabla_TJ\rangle-\langle R(T,J)J|T\rangle]dt=I(J,J),
\end{equation*}
where $J$ is a Jocabi field along $\gamma$ such that $J(0)=0$, $J(r)=X$.
\end{proposition}
\begin{proof}
Let $T(x)$ be the unit vector tangent to the unique geodesic joint $p$ to $x$ and pointing away from $p$. Consider the orthogonal decomposition
$$T_xM=\text{span} \{T(x)\}\oplus T^{\perp}(x),$$
where $T^{\perp}(x)=\{X \in T_xM|\langle T(x),X\rangle=0\}$.
We assert that there are also orthogonal decompositions relative to $H(\rho)$, in the sense that
$$H(\rho)(T(x),T^{\perp}(x))=0.$$
To show this, let $X \in T^{\perp}(x)$, then
$$H(\rho)(T (x),X)=H(\rho)(X,T(x))=XT(\rho)(x)-(\nabla_XT) \rho(x).$$
Since $(\nabla\rho)(x) =T(x)$ and $T(\rho)=\langle \nabla \rho |\nabla \rho \rangle =1$, this implies
$$H(\rho)(T(x),X)=-(\nabla_XT)\rho(x)=-\langle \nabla \rho| \nabla_X T \rangle_x=-\frac{1}{2}X\langle T|T\rangle_x=0.$$ By the definition of Hessian, it follow that
$$H(\rho)(T(x),T(x))=0.$$

Let $X \in T^{\perp}(x)$ and $\zeta:(-\varepsilon,\varepsilon)\rightarrow M$ be a normal geodesic such that $\dot{\zeta}(0)=X$. Let $\gamma_s:[0,r] \rightarrow M$ be a variation of $\gamma$ such that $\gamma_s=$ (the unique geodesic joining $p$ to $\zeta(s)$) and $\gamma_0=\gamma$. Note that

(i) the transversal vector field $J=\frac{d}{ds}(\gamma_s(t))|_{s=0}$ of $\gamma_s$ along $\gamma$ is a Jacobi field;

(ii) $J(0)=0$ and $J(r)=X$;

(iii) $\langle J|\dot{\gamma}\rangle =0.$

We denote $T:=\dot{\gamma}=\nabla \rho.$
Therefore, from Definition \ref{d-3.1} and Corollary \ref{c-3.1}, we have
\begin{equation*}\begin{split}
H(\rho)(X,X)&=X\dot{\zeta}(\rho)|_{s=0}-\nabla_X\dot{\zeta}(\rho)(x)\\
&=\dot{\zeta}\dot{\zeta}(\rho)|_{s=0}-\langle \nabla_{\dot{\zeta}}\dot{\zeta}|\nabla \rho \rangle_x\\
&=\frac{d^2 L}{d^2 s}\Big|_{s=0}- \langle \nabla_{\dot{\zeta}}\dot{\zeta}|\nabla \rho \rangle_x\\
&=I(J,J).
\end{split}\end{equation*}
\end{proof}
\begin{proposition}(\cite{cheeger})\label{p-3.2}
Let $\gamma$ be a geodesic in a Riemannian manifold $(M,g)$ from $p$ to $q$ such that there are no points conjugate to $p$ on $\gamma$. Let $W$ be a piecewise smooth vector field on $\gamma$ and $V$ the unique Jacobi field such that $V(p)=W(p)=0$ and $V(q)=W(q)$. Then $I(V,V) \leq I(W,W)$, and equality holds only if $V=W$.
\end{proposition}

As an application of Proposition \ref{p-3.1} and \ref{p-3.2}, we have the following theorem.
\begin{theorem}\label{t-3.2}
Suppose that $(M,g)$ is a Riemannian manifold with a pole $p$ such that its radial sectional curvature is bounded from below by a negative constant $-K^2$. Suppose that $\gamma:[0,r] \rightarrow M$ is a normal geodesic such that $\gamma(0)=p$ and $\gamma(r)=x$. Then
\begin{equation*}
H(\rho)(u,u)(x) \leq \frac{1}{\rho}+K,
\end{equation*}
where $u= u^i\frac{\partial}{\partial x^i} \in T_xM$ is a unit vector.
\end{theorem}
\begin{proof}
Firstly, let $u \in T_xM$ be a unit vector. By Proposition \ref{p-3.1}, we have
\begin{equation}\label{A_1}
H(\rho)(u,u)=\int_0^r[\langle \nabla_TJ|\nabla_TJ\rangle-\langle R(T,J)J|T\rangle]dt =I(J,J),
\end{equation}
where $J$ is a Jacobi field along the geodesic $\gamma$ such that $J(0)=0$ and $J(r)=u$.

 Let $\eta(t)$ be a unit vector field along $\gamma$ such that $\eta(r)=u$. Set $\xi(t)=(\frac{t}{r})^\alpha\eta(t)$ and $\alpha>1$, then it is clear that $\xi(0)=J(0)=0$ and $\xi(r)=J(r)=u$. By Proposition \ref{p-3.2} and Definition \ref{d-3.1}, we have
\begin{equation}\begin{split}\label{eq-3.1}
I(J,J) &\leq \int_0^r[\langle\nabla_{T}\xi|\nabla_{T}\xi\rangle-\langle R(T,\xi)\xi|T\rangle]dt\\
& \leq \int_0^r\Big[\Big\langle{\alpha}\Big(\frac{t}{r}\Big)^{\alpha-1}\xi\Big|{\alpha}\Big(\frac{t}{r}\Big)^{\alpha-1}\xi\Big\rangle+K^2 \langle T|T\rangle\langle \xi|\xi\rangle\Big]dt\\
&= \int_0^r\Big[\alpha^2\Big(\frac{t}{r}\Big)^{2(\alpha-1)}+K^2\Big(\frac{t}{r}\Big)^{2\alpha}\Big]dt\\
&=\frac{\alpha^2}{2\alpha-1}\cdot\frac{1}{r}+\frac{K^2r}{2\alpha+1}\\
&=\frac{1}{r}+\frac{(\alpha-1)^2}{2\alpha-1}\frac{1}{r}+\frac{K^2r}{2\alpha+1},
\end{split}\end{equation}
where in the second inequality we used the condition that $-K^2$ is the lower bound of the radial sectional curvature.

We can take a suitable $\alpha$ such that
\begin{equation*}
\frac{(\alpha-1)^2}{2\alpha-1}\frac{1}{r}=\frac{K^2r}{2\alpha+1}.
\end{equation*}
Therefore, we obtain
\begin{equation}\label{z-1}
\frac{(\alpha-1)^2}{2\alpha-1}\frac{1}{r}+\frac{K^2r}{2\alpha+1} =2 \sqrt{\frac{K^2(\alpha-1)^2}{4\alpha^2-1}}= \sqrt{\frac{K^2(4\alpha^2-8\alpha+4)}{4\alpha^2-1}} \leq K,
\end{equation}
where in the last inequality we used the fact that $ \alpha>1$.

By $\eqref{eq-3.1}$, $\eqref{z-1}$ and the fact $\rho=r$, we get
\begin{equation}\label{A_2}
I(J,J) \leq \frac{1}{\rho}+K.
\end{equation}
Plugging \eqref{A_2} into \eqref{A_1} we get
$$H(\rho)(u,u)(x) \leq \frac{1}{\rho}+K.$$
\end{proof}

As a simple application, we obtain the following result.
\begin{corollary}\label{c-3.2}
 Suppose that $(M,g)$ is a Riemannian manifold with a pole $p$ such that its radial sectional curvature is bounded from below by a negative constant $-K^2$. Suppose that $\gamma:[0,r] \rightarrow M$ is a normal geodesic such that $\gamma(0)=p$ and $\gamma(r)=x$. Then Then with respect to  the normal coordinates at the point $x$, we have
\begin{equation*}
\frac{\partial^2 \rho}{\partial x^i \partial x^j}(x)\leq \Big(\frac{1}{\rho}+K\Big)\delta_{ij}.
\end{equation*}
\end{corollary}
\begin{proof}
For any given point $x_0 \in M$, there exists a local coordinate system $(x^1,\cdots, x^n)$ in a neighborhood of $x_0$ such that $g_{ij}(x_0)=\delta_{ij}$ and $\Gamma_{ij}^k(x_0)=0$. By definition of Hessian, we have
\begin{equation*}\begin{split}
H(\rho)(u,u)&=H(\rho)\Big(u^i\frac{\partial}{\partial x^i},u^j\frac{\partial}{\partial x^j}\Big)\\
&=u^iu^jH(\rho)\Big(\frac{\partial}{\partial x^i},\frac{\partial}{\partial x^j}\Big)\\
&=u^iu^j\Big(\frac{\partial^2 \rho}{\partial x^i \partial x^j}+\Gamma^k_{ij}(x)\frac{\partial \rho}{\partial x^k}\Big).
\end{split}\end{equation*}
Thus at the point $x_0$, we have
$$H(\rho)(u,u)(x_0)=u^iu^j\frac{\partial^2 \rho}{\partial x^i \partial x^j}(x_0).$$
By Theorem \ref{t-3.2}, we have
$$ \frac{\partial^2 \rho}{\partial x^i \partial x^j}(x_0) \leq \Big(\frac{1}{\rho}+K\Big)\delta_{ij},$$
where $\delta_{ij}$ is the Kronecker symbol.
\end{proof}

By Theorem \ref{t-3.2} and Corollary \ref{c-3.2}, we obtain the following corollary.
\begin{corollary}\label{c-3.3}
 Suppose that $(M,g)$ is a Riemannian manifold with a pole $p$ such that its sectional curvature is bounded from below by a negative constant $-K^2$. Suppose that $\gamma:[0,r] \rightarrow M$ is a normal geodesic such that $\gamma(0)=p$ and $\gamma(r)=x$. Then with respect to  the normal coordinates at the point $x$, we have
$$\frac{\partial^2\rho^2}{\partial x^i \partial x^j}u^iu^j \leq2(2+\rho K), $$
where $u=u^i\frac{\partial}{\partial x^i} \in T_xM$ is a unit vector.
\end{corollary}
\begin{proof}
Note that $\nabla (\rho^2)=2\rho \nabla \rho$. Thus we have
\begin{equation}\label{eq-3.5}
H(\rho^2)(u,u)=2(d\rho(u))^2+2\rho H(\rho)(u,u).
\end{equation}
If we write
$$u=\sum_{i=1}^n(u')^iE_i=\sum_{i=1}^{n-1}(u')^iE_i+(u')^nE_n,$$
we get
\begin{equation}\label{eq-3.6}
(d\rho(u))^2=((u')^n)^2 \leq g(u)=1.
\end{equation}
For any given point $x_0 \in M$, there exists a local coordinate system $(x^1,\cdots,x^n)$ in a neighborhood of $x_0$ such that $g_{ij}(x_0)=\delta_{ij}$ and $\Gamma_{ij}^k(x_0)=0$. By \eqref{eq-3.5}, \eqref{eq-3.6} and Corollary \ref{c-3.2}, at the point $x_0$ we have
\begin{equation*}\begin{split}
H(\rho^2)(u,u)&=2(d\rho(u))^2+2\rho H(\rho)(u,u)\\
&\leq 2(2+\rho K).
\end{split}\end{equation*}
\end{proof}
In differential geometry, the following Gauss lemma is of importance.
\begin{theorem}(\cite{cheeger})\label{t-3.3}
Let $(M,g)$ be a Riemannian manifold, fix $p \in M$. If $\rho(t)=tv$ is a ray through the origin of $T_pM$ and $w\in (T_pM)_{\rho(t)}$ is perpendicular to $\rho'(t)$, then $d\exp(w)$ is perpendicular to $d\exp(\rho'(t))$.
\end{theorem}
Setting
$$B_p(r)=\{x \in M |d(p,x) < r\}, \quad S_p(r)=\{x \in M |d(p,x) = r\},$$
where $d(p,x)$ is the distance from $p$ to $x$ induced by the Riemannian metric $g$ on $M$.

Using theorem \ref{t-3.3}, one can easily get the following corollary.
\begin{corollary}\label{c-3.4}
Let $(M,g)$ be a Riemannian manifold, fix $p \in M$ and $x \in S_p(r)$. Then $u \in T_xM$ belongs to $T_x(S_p(r))$ if and only if
$$\langle u|T \rangle=0,$$
where $T$ is the unit vector tangent to the unique geodesic joint $p$ to $x$ and pointing away from $p$.
\end{corollary}

Let $(M,g)$ be a Riemannian manifold with a pole $p$. We know that $\rho(x)$ is a smooth function on $M\setminus\{p\}$. By the classical Hopf-Rinow theorem, there exist a minimizing geodesic $\sigma$ connecting $p$ to $x$, such that
$$\rho(x)=d(p,x)=L(\sigma).$$
Since the gradient of $\rho^2(x)$ is equal to $2\rho\nabla\rho$. By Corollary \ref{c-3.4} and the fact that $\nabla \rho =T$, we easily have
$$\langle \nabla \rho^2|T\rangle= \langle 2\rho \nabla \rho |T\rangle= 2\rho,$$
where in the last equality we used the fact that $g(T(x))=\langle T|T\rangle=1$.

Thus we have proved the following theorem.
\begin{theorem}\label{t-3.4}
Suppose that $(M,g)$ is a Riemannian manifold with a pole $p$. Let $\gamma:[0,r] \rightarrow M$ be a normal geodesic. Then
$$\langle \nabla \rho^2|T\rangle=2\rho.$$
\end{theorem}

\section{Holomorphic mappings between complex Finsler manifolds}

In this section, we assume that $M$ and $N$ are two complex manifolds of complex dimension $n$ and $m$, respectively. Suppose that $G:T^{1,0}M\rightarrow [0, +\infty)$ and $H:T^{1,0}N\rightarrow [0,+\infty)$ are two strongly pseudoconvex complex Finsler metrics on $M$ and $N$, respectively. Now we consider a non-constant holomorphic mapping $f:M\rightarrow N$. It gives rise to the pull-back metric $f^\ast H$ on $M$. Thus it makes sense to consider the ratio $\frac{f^\ast H}{G}$ outside the zero section of $T^{1,0}M$. More precisely, let $(z;v)=(z^1,\cdots,z^n;v^1,\cdots,v^n)$ be local complex coordinates on $\tilde{M}:=T^{1,0}M\setminus\{\mbox{zero section}\}$ and $(w;\xi)=(w^1,\cdots,w^m,\xi^1,\cdots,\xi^m)$ be local complex coordinates on $\tilde{N}=T^{1,0}N\setminus\{\mbox{zero section}\}$. Then along the map $f$, we have
$$
w^i=f^i(z^1,\cdots,z^n),\quad \xi^i=(f_\ast)^i_z(v)=\frac{\partial f^i}{\partial z^\alpha}v^\alpha
$$
for $i=1,\cdots,m$.
Thus
\begin{equation}
(f^\ast H)(z;v)=H(f(z);(f_\ast)_z(v)).
\end{equation}

Now we define
\begin{equation}\label{u}
u(z;v):=\frac{(f^*H)(z;v)}{G(z;v)},\quad \forall (z;v)\in \tilde{M}.
\end{equation}
By Definition \ref{d-2.1}, we have
 $$u(z;\lambda v)=u(z;v),\quad\lambda\in\mathbb{C}\setminus\{0\},\forall (z;v)\in\tilde{M}.$$
  That is, $u$ is a well-defined non-negative continuous function defined on the projective bundle $PT^{1,0}M$.

Let $U$ be a coordinate neighborhood in $M$. For each $z\in U$ (we also use $z$ to denote local complex coordinates if it causes no confusion), we define
\begin{equation}\label{EQ-AC}
\tilde{u}(z):=\max_{ v\in PT_z^{1,0}M} u(z;v),
\end{equation}
or equivalently
\begin{equation}
\tilde{u}(z)=\max_{v\in S_z^{1,0}M}u(z;v),\label{B}
\end{equation}
where $S_z^{1,0}M=\{v\in T_z^{1,0}M|G(z;v)=1\}$.

 Now we show that $\tilde{u}$ defined by \eqref{EQ-AC}, or equivalently \eqref{B}, is a well-defined function on $M$. In deed, let $U_A$ and $U_B$ be two coordinate neighborhoods in $M$ with complex coordinates $z_A=(z_A^1,\cdots,z_A^n)$ and $z_B=(z_B^1,\cdots,z_B^n)$ respectively such that $\mathcal{U}_A:=\pi^{-1}(U_A)$ and $\mathcal{U}_B:=\pi^{-1}(U_B)$ are two coordinates neighborhoods in $PT^{1,0}M$ with the induced homogeneous coordinates $(z_A;[v_A])$ and $(z_B; [v_B])$ respectively.
  By definition,
 \begin{eqnarray*}
\tilde{u}(z_A)&=&\max_{v_A\in PT_{z_A}^{1,0}M} u(z_A;v_A),\quad
\tilde{u}(z_B)=\max_{v_B\in PT_{z_B}^{1,0}M}u(z_B; v_B).
\end{eqnarray*}
 Now if $PT_{z_A}^{1,0}M$ and $PT_{z_B}^{1,0}M$ represents the same fiber of $PT^{1,0}M$, then by the continuity of $u$, and the compactness of the fibers of $PT^{1,0}M$, we have $\tilde{u}(z_A)=\tilde{u}(z_B)$ on $U_A\cap U_B\neq \emptyset$, which shows that $\tilde{u}$ is a well-defined function on $M$.

More precisely, we have the following theorem.
\begin{theorem}\label{P-1}
Suppose that $(M,G)$ and $(N,H)$ are two strongly pseudoconvex complex Finsler manifolds. Let $f:M \rightarrow N$ be a holomorphic mapping. Then
\begin{equation*}
(f^\ast H)(z;v)\leq \tilde{u}(z)G(z;v),\quad \forall (z;v)\in T^{1,0}M,
\end{equation*}
where  $\tilde{u}$ is a continuous function on $M$.
\end{theorem}

\begin{proof}
By definition of $u$ and $\tilde{u}$, we have
$$
(f^\ast H)(z;v)=u(z;v)G(z;v)\leq \tilde{u}(z)G(z;v),\quad \forall (z,v)\in T^{1,0}M.
$$
The only thing we need to do is to that $\tilde{u}$ is a continuous function on $M$. It suffices to prove that $\tilde{u}$ is continuous at an arbitrary fixed point $z_0\in M$.

In deed, let $z_0 \in M$ be an arbitrary fixed point, such that $U$ is a coordinate neighborhood containing $z_0$ on $M$ with coordinates $z$ and $\mathcal{U}:=\pi^{-1}(U)\cong U\times \mathbb{CP}^{n-1}$ is the induced coordinate neighborhood on $PT^{1,0}M$ such that $(z;v)$ are homogeneous coordinates on $\mathcal{U}$.

By definition we have $\tilde{u}(z_0)=u(z_0;v_0)$ for some point $v_0\in S_{z_0}^{1,0}M$. Note that $u$ is continuous at the point $(z_0;v_0)\in \mathcal{U}$, thus for each $\varepsilon>0$, there exists a $\delta>0$, such that whenever $|z-z_0|<\delta$ and $|v-v_0|<\delta$, we have $|u(z;v)-u(z_0;v_0)|<\varepsilon$.

Let $(z;v)\in \mathcal{U}$ be an arbitrary point which is sufficiently close to $(z_0;v_0)$.
We assume that $\tilde{u}(z)=u(z;w)$ for some point $w\in S_z^{1,0}M$.
Then for a sufficient large number $\mathcal{N}>0$ we have $|\frac{w}{\mathcal{N}}-\frac{v_0}{\mathcal{N}}| <\delta$.
On the other hand, by the definition of $\tilde{u}$, we have
$$\tilde{u}(z)=u\Big(z,\frac{w}{\mathcal{N}}\Big),\quad \tilde{u}(z_0)=u\Big(z_0,\frac{v_0}{\mathcal{N}}\Big).$$
Thus we have
\begin{eqnarray*}
|\tilde{u}(z)-\tilde{u}(z_0)|
&=&\Big|u(z,\frac{w}{\mathcal{N}}\Big)-u\Big(z_0,\frac{v_0}{\mathcal{N}}\Big)\Big|\\
&\leq&\Big|u\Big(z,\frac{w}{\mathcal{N}}\Big)-u\Big(z_0,\frac{w}{\mathcal{N}}\Big)\Big|+\Big|u\Big(z_0,\frac{w}{\mathcal{N}}\Big)-u\Big(z_0,\frac{v_0}{\mathcal{N}}\Big)\Big|\\
&\leq&\varepsilon+\varepsilon\\
&=&2\varepsilon,
\end{eqnarray*}
which shows that $\tilde{u}(z)$ is continuous at the point $z_0$.
\end{proof}

\section{Some lemmas }
\setcounter{equation}{0}
\noindent

In this section, we obtain the estimation of the distance function on a K\"ahler manifold. Now we give some lemmas which are used in the proof of Theorem \ref{th-6.1}.
For our convenience, we also denote a K\"ahler manifold $(M,ds_M^2)$ by $(M,h)$ with $$h:=ds_M^2(v,\overline{v})=h_{\alpha\overline{\beta}}(z)dz^\alpha(v)\overline{dz^\beta(v)}=h_{\alpha\overline{\beta}}(z)v^\alpha\overline{v^\beta},\quad \forall v=v^\alpha\frac{\partial}{\partial z^\alpha}\in T_z^{1,0}M$$ depending on the actual situation.

The following lemma is well-known in Riemannian geometry, which is also a special case of Proposition 2.6.1 in \cite{abate}.
\begin{lemma}\label{b}
Let $(M,h)$ be a Hermitian manifold. If we take $g=\mbox{Re}(h)$ of $h$. Then
\begin{equation*}\begin{split}
g_{ab}u^a_1u^b_2=2h_{\alpha\bar{\beta}}v_1^{\alpha}\overline{v_2^{\beta}},\quad v_1,v_2 \in T^{1,0}_zM, u_j= v_j^\circ \quad \text{for} \quad j=1,2,
\end{split}\end{equation*}
where $^\circ: T^{1,0}M \rightarrow T_{\mathbb{R}}M$ is an $\mathbb{R}$-isomorphism given by $v^\circ=v+\bar{v}=u, \forall v \in T^{1,0}M.$
\end{lemma}
\begin{lemma}(\cite{greene1})\label{l-4.2}
 Let $f$ be a real-valued function on a K\"ahler manifold $M$. Then
$$Lf(u_\circ,u_\circ)=\nabla^2f(u,u) + \nabla^2f(Ju,Ju),$$
where $Lf=4\frac{\partial^2f}{\partial z^\alpha\overline{z}^\beta}dz^\alpha\wedge d\overline{z}^\beta$ and $_\circ:T_{\mathbb{R}}M \rightarrow T^{1,0}M$ is an $\mathbb{R}$-isomorphism given by $v=u_\circ=\frac{1}{2}(u-\sqrt{-1}Ju)$ and $J$ is the canonical complex structure on $M$.
\end{lemma}
\begin{lemma}\label{l-4.3}
 Suppose that $(M,h)$ is a K\"ahler manifold with a pole $p$ such that its radial sectional curvature is bounded from below by a negative constant $-K^2$. Suppose that $\gamma:[0,r] \rightarrow M$ is a normal geodesic such that $\gamma(0)=p,\gamma(r)=z$. Denote $\rho(z)$ the distance function from $p$ to $z$. Then
$$\frac{\partial^2 \rho^2}{\partial z^{\alpha} \partial \bar{z}^{\beta}}v^{\alpha}\bar{v}^{\beta}\leq 2+\rho K,$$
where $v=\frac{1}{2}(u-\sqrt{-1}Ju)$, $J$ is the canonical complex structure on $M$ and $v=v^{\alpha}\frac{\partial}{\partial z^{\alpha}} \in T^{1,0}_zM$ is a unit vector satisfying $h(v)=1$.
\end{lemma}
\begin{proof}
By Lemma \ref{l-4.2}, we know that
\begin{equation*}\begin{split}
4\frac{\partial^2 \rho^2}{\partial z^{\alpha}\bar{z}^{\beta}}v^{\alpha}\bar{v}^{\beta}&=L(\rho^2)
=\nabla^2\rho^2(u,u)+\nabla^2\rho^2(Ju,Ju).
\end{split}\end{equation*}
By Corollary \ref{c-3.3} and Lemma \ref{b}, we have
\begin{equation*}\begin{split}
\frac{\partial^2 \rho^2}{\partial z^{\alpha} \partial \bar{z}^{\beta}}v^{\alpha}\bar{v}^{\beta}
\leq \frac{1}{4}\{2(2+\rho K)g(u)+2(2+\rho K)g(Ju)\}
=2+\rho K,
\end{split}\end{equation*}
where the last step we used the fact that $g(u)=g(Ju)=h(v)=1$.
\end{proof}
\begin{remark}\label{remark1}
If $\triangle$ is the unit disk in $\mathbb{C}$ endowed with the Poincar$\acute{\mbox{e}}$ metric $P$ whose Gaussian curvature is $-4$, and denote $\varrho(\zeta)$ the distance function from $0$ to $\zeta\in\triangle$. Then by Lemma \ref{l-4.3},
\begin{equation}
\frac{\partial^2\varrho^2(\zeta)}{\partial \zeta\partial\overline{\zeta}}\leq 2[1+2\varrho(\zeta)],\quad\forall \zeta\in\triangle.
\end{equation}
\end{remark}

\begin{lemma}\label{l-4.4}
 Suppose that $(M,h)$ is a K\"ahler manifold with a pole $p$. Suppose that $\gamma:[0,r] \rightarrow M$ is a geodesic such that $\gamma(0)=p,\gamma(r)=z, \dot{\gamma}(r)=T(z)$ and $h(T(z))=1$. Then
\begin{equation*}
2\rho(z)=\langle (\nabla \rho^2(z))_\circ,T(z)\rangle,
\end{equation*}
where $\langle\cdot, \cdot\rangle$ is the Hermitian inner product induced by $h$.
\end{lemma}
\begin{proof}
By the fact that $p$ is a pole, we know that $\rho(z)$ is equal to the length of minimizing geodesic connecting $p$ and $z$. Therefore, we have
$$\nabla \rho(z)=T(z)+ \overline{T(z)}. $$
By Theorem \ref{t-3.4} and $h(T(z))=1$, we have
\begin{equation}\label{A_3}
\langle \nabla \rho^2(z)|(T(z))^{\circ} \rangle= \langle 2\rho(z)(T(z)+ \overline{T(z)})|(T(z))^{\circ} \rangle=2\rho(z).
\end{equation}
Using Lemma \ref{b} and \eqref{A_3}, we have
\begin{equation*}\begin{split}
4\rho(z)&=\langle \nabla \rho^2(z)|(T(z))^{\circ} \rangle\\
&= \mbox{Re}[\langle (\nabla \rho^2(z))_\circ, T(z)\rangle]\\
&=\langle (\nabla \rho^2(z))_\circ, T(z)\rangle,
\end{split}\end{equation*}
where the last step we used the fact that $\langle (\nabla \rho^2(z))_\circ, T(z)\rangle$ is real.
This completes the proof.
\end{proof}

\begin{remark}\label{r-5.2}
If $\triangle$ is the unit disk in $\mathbb{C}$ endowed with the Poincar$\acute{\mbox{e}}$ metric $P$ whose Gaussian curvature is $-4$, and $\varrho(\zeta)$ is the distance function from $0$ to $\zeta\in\triangle$. Then by Lemma \ref{l-4.4},
\begin{equation}
2\varrho(\zeta)=\langle (\nabla \varrho^2(\zeta))_\circ,T(\zeta)\rangle.
\end{equation}
\end{remark}

\section{The Proof of Theorem \ref{mth}}

In this section, we prove the main theorem \ref{mth} in this paper. First we establish the following theorem.
\begin{theorem}\label{th-6.1}
 Suppose that $(M,ds_M^2)$ is a K\"ahler manifold with a pole $p$ such that its holomorphic sectional curvature is bounded from below by a constant $K_1$ and its sectional curvature is also bounded from below. Suppose that $(N,H)$ is a strongly pseudoconvex complex Finsler manifold whose holomorphic sectional curvature
is bounded from above by a negative constant $K_2$. Let $f:M\rightarrow N$ be a holomorphic mapping. Then
\begin{equation}\label{eq-6.1}
(f^*H)(z;dz) \leq \frac{K_1}{K_2}ds_M^2.
\end{equation}
\end{theorem}
\begin{proof}
The key point of the proof is the construction of the auxiliary function \eqref{af} and then using maximum principle.

Let $B_a(p)$ be a closed geodesic ball in $(M,h)$ with its center at $p$ and radius of $a\in(0,+\infty)$. Let $\Delta$ be a unit disk with Poincar\'e metric $P$.
And let $B_b$ be a closed geodesic ball in $(\Delta, P)$ with its center at $0$ and radius of $b\in(0,+\infty)$. We denote the distance function from $0$ to $\zeta$ on $\Delta$ by $\varrho(\zeta)$, and the distance function from $p$ to $z$ on $M$ by $\rho(z)$. Suppose that $\varphi$ is any holomorphic mapping from $\Delta$ into $M$ such that $\varphi(B_b) \subset B_a(p)$ and $h(\varphi'(\zeta))\equiv1$.

The pull-back metric on $\Delta$ of $h$ on $M$ by the holomorphic mapping $\varphi:\Delta\rightarrow M$ is given by
 \begin{equation*}
 (\varphi^*h)(\zeta)= \lambda^2(\zeta)d\zeta d\bar{\zeta}.
 \end{equation*}
 Here we have denoted
 \begin{equation*}
 \lambda^2(\zeta):=h(\varphi(\zeta);\varphi'(\zeta)).
 \end{equation*}

 Note that since $h(\varphi'(\zeta))\equiv1$, we have $\varphi'(\zeta)\neq 0$. Thus $\lambda(\zeta) >0$.
 Now let
 \begin{equation*}
 (f\circ \varphi)^*H(\zeta) =\sigma^2(\zeta) d\zeta d\bar{\zeta}
 \end{equation*}
 be the pull-back metric on $\Delta$ of $H$ on $N$ by the holomorphic mapping $f\circ \varphi:\Delta\rightarrow N$. Here, we have denoted
 \begin{equation*}
 \sigma^2(\zeta):= H((f\circ\varphi)(\zeta);(f\circ\varphi)'(\zeta)).
 \end{equation*}

 i) If $(f\circ\varphi)'(\zeta)= 0$, then \eqref{eq-6.1} holds obviously with $\sigma(0)=0$.

 ii)  Suppose that $(f\circ\varphi)'(\zeta)\neq 0$.

 We define the following auxiliary function:
\begin{equation}
\Phi(\zeta):=[a^2-\rho^2(\varphi(\zeta))]^2[b^2-\varrho^2(\zeta)]^2\frac{\sigma^2(\zeta)}{\lambda^2(\zeta)}.\label{af}
\end{equation}
It is clear that $\Phi(\zeta)\geq 0$ for any $\zeta\in B_b$.

 Note that since the point $p$ is a pole of the K\"ahler manifold $M$,  it follows that  $\rho(z)$ is a smooth function on $M$. Thus the function $\Phi(\zeta)$ defined by \eqref{af} is  smooth for $\zeta\in B_b$. Moreover, $\Phi(\zeta)$ attains its maximum at some interior point $\zeta=\zeta_0\in B_b$ since $\Phi(\zeta) \rightarrow 0$ as $\varrho(\zeta)\rightarrow b$, or equivalently $\zeta$ tends to the boundary  $\partial B_b$ of $B_b$. Thus it suffice for us to seek an upper bound of $\Phi(\zeta_0)$ for an arbitrary $\varphi(\zeta)$ satisfying $(f\circ\varphi)'(\zeta)\neq 0$. We want to use the maximum principle.

 In the following, in order to abbreviate expression of formulas. We denote $\varphi(\zeta)=z=(z^1,\cdots,z^n),\varphi(\zeta_0)=z_0=(z_0^1,\cdots,z_0^n)$, $\varphi'(\zeta)=v=(v^1,\cdots,v^n)$ and $\varphi'(\zeta_0)=v_0=(v_0^1,\cdots,v_0^n)$.

Since $\Phi(\zeta)$ is smooth for $\zeta\in B_b$ and attains its maximum at the interior point $\zeta_0\in B_b$, it necessary that at the point $\zeta=\zeta_0$:
\begin{equation}
0=\frac{\partial}{\partial \zeta} \log \Phi(\zeta)\quad\mbox{and}\quad
0 \geq \frac{\partial^2}{\partial \zeta \partial \bar{\zeta}} \log\Phi(\zeta).\label{sod}
\end{equation}

Substituting \eqref{af} in the second inequality in \eqref{sod}, we have
\begin{equation}\begin{split}\label{EQ-6.2}
0 \geq &2\frac{\partial^2}{\partial \zeta \partial \bar{\zeta}}\log[a^2-\rho^2(z)]
+\frac{\partial^2 }{\partial \zeta \partial \bar{\zeta}}\log \sigma^2(\zeta)-\frac{\partial^2 }{\partial \zeta \partial \bar{\zeta}}\log\lambda^2(\zeta)+2\frac{\partial^2 }{\partial \zeta \partial \bar{\zeta}}\log [b^2-\varrho^2(\zeta)]\\
=&-2[a^2-\rho^2(z)]^{-1}\frac{\partial^2\rho^2(z)}{\partial z^\alpha\partial\bar{z}^\beta}v^{\alpha}\bar{v}^{\beta}-2[a^2-\rho^2(z)]^{-2}\Big|\frac{\partial\rho^2(z)}{\partial z^\alpha}v^{\alpha}\Big|^2\\
&+\frac{\partial^2 }{\partial \zeta \partial \bar{\zeta}}\log \sigma^2(\zeta)-\frac{\partial^2 }{\partial \zeta \partial \bar{\zeta}}\log \lambda^2(\zeta)\\
&-2[b^2-\varrho^2(\zeta)]^{-1}\frac{\partial^2 \varrho^2(\zeta)}{\partial \zeta \partial \bar{\zeta}}-2[b^2-\varrho^2(\zeta)]^{-2}\Big|\frac{\partial \varrho^2(\zeta)}{\partial \zeta}\Big|^2.
\end{split}\end{equation}
By Lemma \ref{l-2.1}, \ref{l-2.2} and the curvature assumptions of $h$ and $H$ in Theorem \ref{th-6.1}, we have
\begin{equation}
\frac{\partial^2 }{\partial \zeta \partial \bar{\zeta}}\log \sigma^2(\zeta)\geq -2K_2\sigma^2(\zeta),\quad \frac{\partial^2 }{\partial \zeta \partial \bar{\zeta}}\log \lambda^2(\zeta)\leq -2K_1\lambda^2(\zeta).\label{cc1}
\end{equation}

By Lemma \ref{l-4.3}, we have
\begin{equation}\begin{split}\label{EQ-6.3}
\frac{\partial^2\rho^2(z)}{\partial z^\alpha\partial \bar{z}^\beta}v^{\alpha}\bar{v}^{\beta}
 \leq 2+\rho K
\leq 2+aK,
\end{split}\end{equation}
where in the last step we used the inequality $\rho(z) \leq a$.

By the Remark \ref{remark1}, at $\zeta=\zeta_0$, we have
\begin{equation}\label{A-7}
\frac{\partial^2 \varrho^2(\zeta)}{\partial \zeta \partial \bar{\zeta}} \leq 2[1+2\varrho(\zeta_0)]\leq 2(1+2b)
\end{equation}
since $\varrho(\zeta_0)\leq b$.

In order to estimate the first order term of $\rho^2(z)$ and $\varrho^2(\zeta)$ in \eqref{EQ-6.2}, we use normal coordinates. Since $M$ is a K\"ahler manifold, we can choose coordinates around $z_0$ such that at the point $z_0$, we have
$$h_{\alpha \bar{\beta}}(z_0)=\delta_{\alpha \beta}, \quad 1 \leq \alpha, \beta \leq n.$$

Thus by Lemma \ref{l-4.4},  at $\zeta=\zeta_0$, we have
\begin{equation}\begin{split}\label{EQ-6.4}
\Big|\frac{\partial\rho^2(z_0)}{\partial z^\alpha}v_0^{\alpha}\Big|
&\leq |\langle (\nabla(a^2-\rho^2(z_0)))_\circ,T(z_0)\rangle|\\
&=|\langle (\nabla(\rho^2(z_0)))_\circ,T(z_0)\rangle|=2\rho(z_0) \\
&\leq 2a,
\end{split}\end{equation}
where the last step we used the fact that $\rho(z_0) \leq a$.
For the same reasons as \eqref{EQ-6.4}, that is, by the Remark \ref{r-5.2},  at $\zeta=\zeta_0$, we have
\begin{equation}\label{A-8}
\Big|\frac{\partial \varrho^2(\zeta)}{\partial \zeta}\Big| \leq 2b.
\end{equation}

Substituting \eqref{cc1}, \eqref{EQ-6.3}, \eqref{A-7}, \eqref{EQ-6.4},  \eqref{A-8} into \eqref{EQ-6.2}, we have (at $\zeta=\zeta_0$),
\begin{equation*}
0 \geq -K_2\sigma^2(\zeta_0)+K_1\lambda^2(\zeta_0)-\frac{2+aK}{a^2-\rho^2(z_0)}-\frac{2a^2}{[a^2-\rho^2(z_0)]^2}-
\frac{2(1+2b)}{b^2-\varrho^2(\zeta_0)}-\frac{2b^2}{[b^2-\varrho^2(\zeta_0)]^2}.
\end{equation*}
Rearranging terms, we get
\begin{equation*}\begin{split}
&\frac{1}{\lambda^2(\zeta_0)}\Big\{(2+aK)[a^2-\rho^2(z_0)][b^2-\varrho^2(\zeta_0)]^2+2a^2[b^2-\varrho^2(\zeta_0)]^2\\
&+2(1+2b)[b^2-\varrho^2(\zeta_0)][a^2-\rho^2(z_0)]^2+2b^2[a^2-\rho^2(z_0)]^2\Big\}-K_1[a^2-\rho^2(\zeta_0)]^2[b^2-\varrho^2(\zeta_0)]^2 \\
&\geq -K_2\frac{\sigma^2(\zeta_0)}{\lambda^2(\zeta_0)}[a^2-\rho^2(\zeta_0)]^2[b^2-\varrho^2(\zeta_0)]^2=-K_2\Phi(\zeta_0)\\
&\geq-K_2 \Phi(\zeta)=-K_2[a^2-\rho^2(\varphi(\zeta))]^2[b^2-\varrho^2(\zeta)]^2\frac{\sigma^2(\zeta)}{\lambda^2(\zeta)}
\end{split}\end{equation*}
for any $\zeta\in B_b$. That is, we have
\begin{equation*}\begin{split}
&\frac{1}{\lambda^2(\zeta_0)}\Big\{(2+aK)[a^2-\rho^2(z_0)][b^2-\varrho^2(\zeta_0)]^2+2a^2[b^2-\varrho^2(\zeta_0)]^2\\
&+2(1+2b)[b^2-\varrho^2(\zeta_0)][a^2-\rho^2(z_0)]^2+2b^2[a^2-\rho^2(z_0)]^2\Big\}-K_1[a^2-\rho^2(\zeta_0)]^2[b^2-\varrho^2(\zeta_0)]^2 \\
&\geq-K_2[a^2-\rho^2(\varphi(\zeta))]^2[b^2-\varrho^2(\zeta)]^2\frac{\sigma^2(\zeta)}{\lambda^2(\zeta)}
\end{split}\end{equation*}
for any $\zeta\in B_b$. Now divided by $a^4b^4$ on both side of the above inequality and then letting $a \rightarrow +\infty$ and $b \rightarrow +\infty$, respectively, we obtain
$$\frac{\sigma^2(\zeta)}{\lambda^2(\zeta)} \leq \frac{K_1}{K_2}$$
 for any holomorphic mapping $\varphi$ from $\Delta$ into $M$ satisfying $(f\circ\varphi)'(\zeta)\neq 0$. By the arguments of i) and ii), it follows that
$$(f^*H)(z;dz) \leq \frac{K_1}{K_2}ds_M^2.$$
This completes the proof.
\end{proof}


\begin{theorem}\label{t-1.3}
 Suppose that $(M,ds_M^2)$ is a complete K\"ahler manifold with holomorphic sectional curvature bounded from below by a constant $K_1$ and sectional curvature bounded from below, while $(N,H)$ is a strongly pseudoconvex complex Finsler manifold with holomorphic sectional curvature of the Chern-Finsler connection bounded from above by a constant $K_2<0$. Then any holomorphic map $f$ from $M$ into $N$ satisfies
\begin{equation}
(f^*H)(z;dz) \leq \frac{K_1}{K_2}ds_M^2.
\end{equation}
\end{theorem}

\begin{proof}
If $(M,ds_M^2)$ is a complete K\"ahler manifold without cut points, then all the Lemmas \ref{l-2.1}, \ref{l-2.2},\ref{l-4.3},  \ref{l-4.4} used in the proof of Theorem \ref{th-6.1}
 still hold. In this case, Theorem \ref{t-1.3} holds obviously.

If $(M,ds_M^2)$ is a complete K\"ahler manifold with cut points. The proof essentially goes the same lines as that in Chen-Cheng-Lu \cite{chen}. We use the notations in Theorem \ref{th-6.1} and its proof. Let $p$ be an arbitrary point. And $\varPhi(\zeta)$ attains its maximum value at $\zeta_0$, that is, $\varphi(\zeta_0)=z_0$. Since $(M,ds_M^2)$ is complete, thus by Hopf-Rinow theorem for a Riemannian metric, there exists a minimizing geodesic $\gamma:[0,1]\rightarrow M$ joining $p$ and $z_0$ such that $\gamma(0)=p$ and $\gamma(1)=z_0$. If there is a $t_0\in(0,1)$ such that $\gamma(t_0)=p_1$ is the first cut point to the point $z_0$ along the inversely directed geodesic $\gamma_1:=\gamma(1-t)$ for all $t\in[0,1]$. Let $\varepsilon>0$ be a sufficient small given number, then it is clear that $\gamma(t_0+\varepsilon)$ is not a cut point of $z_0$ with respect to the geodesic $\gamma_1$. Define $\tilde{\rho}(p,z):=\rho(p,\gamma(t_0+\varepsilon))+\rho(\gamma(t_0+\varepsilon),z)$. Then, using the triangle inequality, we have
$$
\rho(p,z)\leq \tilde{\rho}(p,z)\quad\mbox{and}\quad \rho(p,z_0)=\tilde{\rho}(p,z_0).
$$
So that $$[a^2-\tilde{\rho}^2(\varphi(\zeta))]^2[b^2-\varrho^2(\zeta)]^2\frac{\sigma^2(\zeta)}{\lambda^2(\zeta)}$$ is smooth at the point $\zeta_0$ and we have
$$
[a^2-\tilde{\rho}^2(\varphi(\zeta))]^2[b^2-\varrho^2(\zeta)]^2\frac{\sigma^2(\zeta)}{\lambda^2(\zeta)}\leq\varPhi(\zeta)$$
 and
 $$[a^2-\tilde{\rho}^2(\varphi(\zeta_0))]^2[b^2-\varrho^2(\zeta_0)]^2\frac{\sigma^2(\zeta_0)}{\lambda^2(\zeta_0)}=\varPhi(\zeta_0).$$
 Now  by passing the discussion of $\tilde{\rho}(p,z)$ to $\rho(\gamma(t_0+\varepsilon),z)$, the remaining proof goes the same lines as Theorem \ref{th-6.1}. This completes the proof.
\end{proof}

\begin{remark}
There are many examples of complete K\"ahler manifolds. By Lempert's fundamental results in \cite{lempert}, the Kobayashi metrics on any bounded strongly convex domains in $\mathbb{C}^n$ with smooth boundaries are strongly pseudoconvex complex Finsler metrics with constant holomorphic sectional curvature $-4$. These serve as the most important and nontrivial examples satisfying our assumption of the target manifold $N$ in Theorem \ref{th-6.1}.
\end{remark}

  In general, however, even on strongly convex domains in $\mathbb{C}^n$ with smooth boundaries, the Kobayashi metrics are not explicitly given and uncomputable. The following theorem shows that every bounded domain in $\mathbb{C}^n$ can be endowed with an explicit non-Hermitian quadratic strongly pseudoconvex complex Finsler metric with holomorphic sectional curvature bounded above by a negative constant.

\begin{theorem}\label{example}
Suppose that $D\subset \mathbb{C}^n(n\geq 2)$ is a bounded domain. Then $D$ admits a non-Hermitian quadratic strongly pseudoconvex complex Finsler metric $G:T^{1,0}D\rightarrow \mathbb{R}^+$ such that its holomorphic sectional curvature is bounded from above by a negative constant.
\end{theorem}
\begin{proof}
Without lose of generality we assume that $D$ contains the origin $0\in\mathbb{C}^n$.
 Set
 $$M_0:=\sup_{z\in D}\{\|z\|\}>0,$$
 here $\|\cdot\|$ denotes the canonical complex Euclidean norm of $z\in D$.
For any $z\in D, v\in T_z^{1,0}D$, we denote
$$r:=\langle v, v \rangle=\|v\|^2,\quad t:=\langle z, z \rangle=\|z\|^2, \quad s:=\frac{|\langle z,v\rangle|^2}{r},$$
where $\langle\cdot,\cdot\rangle$ denotes the canonical complex Euclidean inner product in $\mathbb{C}^n$.

For every constants $a,b$ satisfying $0\neq a\in\mathbb{R}$ and $b<\frac{1}{M_0}$, we define
\begin{equation}
G(z,v):=r\phi(t,s)\quad\mbox{with}\quad \phi(t,s)=e^{at+bs},\quad \forall z\in D, v\in T_z^{1,0}D.\label{G}
\end{equation}
It is clear that $G:T^{1,0}D\rightarrow \mathbb{R}^+$ is a non-Hermitian quadratic complex Finsler metric on $D$.
By Proposition 2.6 in \cite{zhong}, it is easy to check that the function $G$ defined in \eqref{G} is a strongly pseudoconvex complex Finsler metric on $D$. Actually it is easy to check that $G$ is a strongly pseudoconvex weakly complex Berwald metric on $D$  satisfying $g(t)\equiv 0$ in (3.23) of Theorem 3.4 in \cite{zhong}.

By Remark 3.9 in \cite{zhong}, the holomorphic sectional curvature of $G$ is given by
$$K_G(z,v)=-\frac{2(a+b)}{\phi},\quad \forall 0\neq v\in T_z^{1,0}D.$$
Thus if $a+b>0$,
 then $K_G$ is negative for any $z\in D$ and $0\neq v\in T_z^{1,0}D$. Now since $t \geq 0$ and $s=\frac{|\langle z,v\rangle|^2}{r} \leq t$,  we have $\phi(t,s)= e^{at+bs} \leq e^{(a+b)t}\leq c$ for a positive constant $c:=e^{(a+b)M_0}$ since $D$ is a bounded domain. Therefore $K_G$  is bounded above by a negative constant, that is,
 $$K_G(z,v)\leq -\frac{2(a+b)}{c},\quad \forall z\in D, 0\neq v\in T_z^{1,0}D.$$
\end{proof}

\section{Some applications of the Schwarz lemma}

In this section, we give some applications of Theorem \ref{th-6.1} and \ref{t-1.3}.
The following corollaries are direct consequences of Theorem \ref{th-6.1} and \ref{t-1.3},respectively.

\begin{corollary}\label{c-6.1}
 Suppose that $(M,ds_M^2)$ is a K\"ahler manifold with a pole $p$ such that its holomorphic sectional curvature is non-negative and its sectional curvature is bounded from below. Suppose that $(N,H)$ is a strongly pseudoconvex complex Finsler manifold whose holomorphic sectional curvature satisfies $K_H \leq K_2$ for a constant $K_2<0$. Then any holomorphic mapping $f$ from $M$ into $N$ is a constant.
\end{corollary}

\begin{corollary}\label{ca}
Suppose that $(M,h)$ is a complete K\"ahler manifold such that its holomorphic sectional curvature is non-negative and its  sectional curvature is bounded from below. Suppose that $(N,H)$ is a strongly pseudoconvex complex Finsler manifold whose holomorphic sectional curvature satisfies $K_H \leq K_2$ for a constant $K_2<0$. Then any holomorphic mapping $f$ from $M$ into $N$ is a constant.
\end{corollary}

Recently in Hermitian geometry, Yang and Zheng \cite{yang} found some partial evidences to the Conjecture 1.6 (d) in \cite{yang}. To state their results, we need the following definition of holomorphic fibration and generic fiber.
\begin{definition} (\cite{zheng})
 A holomorphic mapping $f : M \rightarrow N$ between two compact complex manifolds is called a holomorphic fibration, if it is surjective and connected (i.e.,$f^{-1}(w)$ is connected for any $w \in M$). For any regular value $w$ of $f$, if the fiber $f^{-1}(w)$ is a compact complex submanifold of $M$, we will call such a fiber a generic fiber of $f$.
\end{definition}

Next let's recall the Hartogs phenomenon on a complex manifold.
A fundamental property of holomorphic function on a domain in $\mathbb{C}^n$ with $n\geq 2$ is  the  Hartogs phenomenon which states that a holomorphic function defined on a spherical shell
\begin{equation*}
D^n_{a,b}=\{z \in \mathbb{C}^n|0\leq a <\|z\|^2<b\}
\end{equation*}
can be extended to the entire ball $B^n_b$ (of radius $b$ centered at the origin). In general, one has the following definition.
\begin{definition}(\cite{zheng})
A complex manifold $M$ is said to obey the Hartogs phenomenon, if for any $1>a>0$, any holomorphic mapping from $D^2_{a,1}$ into $M$ can be extended to a holomorphic mapping from the unit ball $B^2$ into $M$.
\end{definition}
In \cite{zheng}, Griffiths and Shiffman proved the following theorem.
\begin{theorem}(\cite{zheng})
Any complete Hermitian manifold with nonpositive holomorphic sectional curvature obeys the Hartogs phenomenon.
\end{theorem}
In 2013, Shen and Shen \cite{shen} proved the following theorem.
\begin{theorem}\label{T-3.3}(\cite{shen})
Any complete strongly pseudoconvex Finsler manifold with non-positive holomorphic sectional curvature obeys the Hartogs phenomenon.
\end{theorem}
Yang and Zheng \cite{yang} obtained the following rigidity theorem.
\begin{theorem}(\cite{yang})\label{zheng}
Let $M$ be a compact Hermitian manifold of complex dimension $n$ with quasi-negative real bisectional curvature. Let $N$ be a compact complex manifold of complex dimension $n$ which admits a holomorphic fibration
$f:N \rightarrow Z$, where a generic fiber is a compact K\"ahler manifold with $c_1=0$. Then $M$ cannot be bimeromorphic to $N$.
\end{theorem}
The key step in the proof of Theorem \ref{zheng} is an application of the Hartogs phenomenon \cite{zheng} and the Schwarz lemma (Theorem 4.5 in \cite{yang}). By Theorem 5.2 in \cite{shen}, we know that any complete complex Finsler manifolds with non-positive holomorphic sectional curvature obeys the Hartogs phenomenon. The curvature conditions on $(M,g)$ in Theorem 4.5 in \cite{yang} is that the second (Chern) Ricci curvature is bounded from below by a negative constant $-b$ and the Ricci curvature is bounded from below.
Combining Theorem 5.2 in \cite{shen} with Corollary \ref{c-6.1}, we have the following theorem.

\begin{theorem}\label{bm}
Let $M$ be a compact complex manifold of $n=\dim_{\mathbb{C}}M\geq 2$ which admits a strongly pseudoconvex Finsler metric $G$  with negative  holomorphic sectional curvature. Let $N$ be a compact complex manifold of the same complex dimension $n$ which admits a holomorphic fibration $f:N \rightarrow Z$, where a generic fiber is a compact K\"ahler manifold with holomorphic sectional curvature bounded from below by a nonnegative constant. Then $M$ cannot be bimeromorphic to $N$.
\end{theorem}

\begin{proof}
The proof essentially goes the same lines as in \cite{yang}. Assume on the contrary that there is a bimeromorphic map $f$ from $N$ into $M$. Since $M$ is compact and the holomorphic sectional curvature $K_G$ of $G$ is negative, by Theorem \ref{T-3.3}, $M$ obeys the
Hartogs phenomenon. So any meromorphic map into $M$ must be holomorphic. Let $U \subseteq M$ be the open set where the holomorphic sectional
curvature $K_G$ of $G$ is negative. Since $M$ is compact and $K_G$ is negative, it follows that $K_G$ is bounded from above by a negative constant. Let $Y$ be a generic fiber of $N$ such that the restriction map $f|_Y: Y \rightarrow M$ is a non-constant map and its image intersect $U$. By assumption, $Y$ is a compact K\"ahler manifold whose holomorphic sectional curvature is bounded from below by a non-negative constant. Thus by  Corollary \ref{c-6.1},  $f|_Y$ is necessary a constant map, which is a contradiction. So that $M$ cannot be bimeromorphic to $N$.
\end{proof}

\begin{remark} In the above theorem, the assumption of the curvature condition on $(M,G)$ is somewhat stronger than that of Theorem \ref{zheng}. The requirement of the generic fiber $f:N\rightarrow M$ satisfying $c_1=0$ in Theorem \ref{zheng} is replaced with the condition that the holomorphic sectional curvature of the K\"ahler metric of the generic fiber $f:N\rightarrow M$ is bounded from below by a non-negative constant.

We exclude the case $\dim_{\mathbb{C}}M=1$ in Theorem \ref{bm} because any strongly pseudoconvex complex Finsler metric $G$ on a complex manifold $M$ of $\dim_{\mathbb{C}}M=1$ is necessary a Hermitian metric. In this case, Theorem \ref{bm} says nothing more than Theorem \ref{zheng}.
\end{remark}

{\bf Acknowledgement:}{The authors thank the referees for carefully reading the manuscript and their valuable corrections/suggestions which improved the presentation of the paper. This work was supported by National Natural Science Foundation of China (Grant No. 12071386, No. 11671330,  No. 11971401).}

\end{document}